\documentclass[11pt]{article}
\usepackage{geometry}
\usepackage{amsmath,amsfonts,amssymb, tikz, amsthm}
\usepackage{hyperref}
\usepackage{authblk}

\usepackage{caption}
\usepackage{graphicx}
\usepackage[english]{babel}
\usepackage[tableposition=top]{caption} 
\usepackage{amsmath}
\usepackage{mathrsfs}
\usepackage{amssymb}
\usepackage{amsfonts}
\usepackage{mathtools}
\usepackage{tikz-cd}
\usepackage[utf8]{inputenc}

\newtheorem{theorem}{Theorem}[section]

\newtheorem{lemma}[theorem]{Lemma}

\newtheorem{proposition}[theorem]{Proposition}

\newtheorem*{setting S}{Setting (S)}

\theoremstyle{definition}
\newtheorem*{notation}{Notation}

\newtheorem*{rk}{Remark}
\theoremstyle{remark}

\numberwithin{equation}{section}

\newcommand{\Z}{\mathbb{Z}}
\newcommand{\Q}{\mathbb{Q}}
\newcommand{\C}{\mathbb{C}}
\newcommand{\R}{\mathbb{R}}

\newcommand{\cl}{\operatorname{Cl}_3}

\newcommand{\mco}{\mathcal{O}}

\newcommand{\keywords}[1]{\noindent\textbf{Keywords:} #1}
\newcommand{\subjclass}[2]{\noindent\textbf{#1 Mathematics Subject Classification:} #2.}

\title{There are consecutive cubic fields with large class numbers, \\ when ordered by discriminant}
\author[1]{V\'{i}t\v{e}zslav Kala}
\author[2]{Om Prakash}
\affil[1,2]{Charles University, Faculty of Mathematics and Physics, Department of Algebra, Sokolovsk\'{a} 83, 186 75 Praha 8, Czech Republic}
\affil[1]{E-mail: vitezslav.kala@matfyz.cuni.cz. ORCID: 0000-0001-5515-6801}
\affil[2]{E-mail: prakash@karlin.mff.cuni.cz. ORCID: 0009-0007-2124-9736}

\begin{document}
\maketitle

\begin{abstract} 
We consider cubic number fields ordered by their discriminants, and show that
there exist arbitrarily long sequences that contain only fields with class numbers greater than a given bound.

\medskip
	 
\medskip
\subjclass{2020}{11R29, 11R16, 11B05, 11B25, 11R45}

\medskip

\noindent\textbf{Funding:}
Both authors were supported by {Czech Science Foundation} grant 26-20514S. 
O.P. was supported by {Charles University} programmes PRIMUS/24/SCI/010 and PRIMUS/25/SCI/017, SVV-2023-260721, and GAUK project No. 252931.
\end{abstract}

\section{Introduction}

Understanding the behavior and distribution of the class number $h(K)$ of a  number field $K$ is one of the fundamental problems in number theory. 
This study has a rich history that begins with Gauss who investigated the class number of quadratic forms. 
In $1801$, Gauss conjectured that $h(\Q(\sqrt{d}))\rightarrow \infty$ as $d\rightarrow -\infty$, and posed the problem of finding all imaginary quadratic fields with a given class number. Heilbronn established the growth of $h(\Q(\sqrt{d}))$ in 1934 and then Heegner characterized all imaginary quadratic fields with class number one in 1952. Now, thanks to the work of Watkins \cite{wat}, we know all imaginary quadratic fields with class numbers up to $100$. 

The class numbers of real quadratic number fields are less well understood than those of their imaginary counterparts. 
The principal tool is the class number formula
\[
h(K) = \frac{w_K \cdot \sqrt{|\Delta_K|}}{2^{r_1} (2\pi)^{r_2} \text{Reg}_K} \cdot \left| \text{Res}_{s=1} \zeta_K(s) \right|,
\]
where $r_1,r_2$ represents the number of real and complex embeddings, $w_K$ denotes the number of roots of unity in $K$, $\text{Reg}_K$ is the regulator, and $\zeta_K$ is the Dedekind zeta function of $K$.  
The first difficulty in applying the class number formula 
is that the unpredictable behavior of the regulator makes it difficult to understand the size of $h(K)$.

In particular, another conjecture of Gauss that there are infinitely many real quadratic fields with class number $1$ is still an open problem. 
It is expected that the size of the regulator is fairly often around $\sqrt{d}$; consequently, $h(K)$ should be often small. Further, it seems counterintuitive that one could find an arbitrarily long sequence of consecutive real quadratic fields, all having large class numbers. However, recently,  Cherubini--Fazzari--Granville--Kala--Yatsyna \cite{che+} showed that for a given positive integer $k$ and $X\in\R_{>0}$, there are $\geq X^{1/2-o(1)}$ integers $d\leq X$ such that all real quadratic fields $\Q(\sqrt{d+1}), \Q(\sqrt{d+2}),\dots, \Q(\sqrt{d+k})$ have class numbers essentially as large as possible.

Let $K$ be a general cubic field. The discriminant $\Delta_K$ of $K$ takes the form $df^2, 9df^2, \text{ or }81df^2$, where $d$ is a fundamental discriminant (i.e., $d=1$ or $d$ is the discriminant of a quadratic field) and $f$ is a squarefree positive integer that is coprime to $3$ \cite[Satz $6$]{Has}. The cubic field $K$ is Galois precisely when $\Delta_K=f^2$. It is clear that not every integer $n$ arises as the discriminant of a cubic field. In particular, for $p>3$, there are no cubic discriminants divisible by $p^3$. Throughout, when we state that $\Delta$ is a \emph{cubic discriminant}, we mean that there is a cubic field whose discriminant is equal to $\Delta$.

A number of results on class numbers of cubic fields are known in special families, such as simplest cubic fields or pure cubic fields. Working with Shanks' family \cite{sha} of simplest abelian cubic fields, Duke \cite{Duk} showed that there exists an absolute constant $c>0$ such that 
\[
h(K)>c\sqrt{\Delta_K} (\log\log \Delta_K/\log \Delta_K)^2
\]
holds for infinitely many abelian cubic fields $K$.
Recently, Dousselin \cite{dou}, using \cite{gs}, showed that there are at least $X^{1/4-o(1)}$ abelian cubic fields in Shanks' family with discriminants $\Delta_K\leq X$ such that 
\[
h(K)\geq \left(\frac{4}{91}e^{2\gamma}+o(1)\right)\sqrt{\Delta_K} (\log\log \Delta_K/\log \Delta_K)^2.
\]
Moreover, in the same paper, he proved that there exist arbitrarily long sequences of abelian cubic fields whose discriminants are close, with class numbers essentially as large as possible. Also recently, Yhee--Byeon \cite{dy} studied the class numbers of consecutive pure cubic fields and established that for a given $k>0$, there are at least $X^{1/3-o(1)}$ integers $1<d<X$ such that $h(\Q(\sqrt[3]{d+j}))$ for all $1\leq j\leq k$ are large. This work was then extended to higher degrees by Das--Krishnamoorthy \cite{dk}.

Inspired by these results, we study class numbers of cubic fields with consecutive discriminants, without imposing  any restrictions on the cubic fields, such as requiring them to be Galois, totally real, or to belong to specific families. As our main result, we prove the following theorem. Throughout the article, $\kappa\approx 0.3193\dots$ is the constant from \cite{ck}.

\begin{theorem}\label{main:thm}
    Fix a sign choice $\pm$, $\varepsilon>0$, and positive integers $k, H$. There exists a positive constant $c$ (depending on $\pm, \varepsilon, k, H$) such that for all sufficiently large $X$, there are at least $cX^{1-\kappa-\varepsilon}$ positive integers $d\leq X$ such that the following hold:
    \begin{enumerate}
        \item For each $1\leq i\leq k$, if there exists a cubic field $K$ with $\pm\Delta_K=d+i$, then $h(K)>H$.

        \item We have 
        \[
        \#\{K \text{ cubic field }\mid \pm \Delta_K\in\{d+1,d+2,\dots,d+k\}\}\geq \frac{k}{2}.
         \]
         
    \end{enumerate}
\end{theorem}

Note that here and throughout the paper, we fix a choice of the sign $\pm$ of the cubic discriminants. Thus, when we say that $\pm\Delta_K=d+i$, we mean that either $+\Delta_K=d+i$, or $-\Delta_K=d+i$, depending on the choice. 

We will prove this result in Theorem \ref{thm:final}. Our proof is constructive; in fact, we first construct integers $0\leq a< m$ such that whenever there exists a cubic field $K$ with $\pm\Delta_K=d+i$, where $d\equiv a\pmod m$, then $h(K)>H$. We obtain this by controlling the \textit{genus number} which gives a lower bound for the size of the 2-torsion in the class group of $K$. The reason for this approach is that it sidesteps the difficult issue of estimating the regulator (thanks to which we do not need to restrict ourselves to a suitable family of cubic fields). At the same time, it makes it impossible for us to obtain class numbers that would be as large as allowed by the class number formula -- we only show that the class numbers are larger than some prescribed bound.
We then show that indeed there are many cubic fields with $ \pm \Delta_K\in\{d+1,d+2,\dots,d+k\}$, which also ensures that our first assertion is not vacuous.

Throughout this article, we count cubic fields up to isomorphism. 
Note that if we view cubic fields as subfields of $\C$, then 
a given cubic field $K$ has either 1 or 3 isomorphic copies in $\C$, depending on whether $K$ is Galois or not (but in either case, we count such copies as only 1 cubic field).

Davenport--Heilbronn \cite{dh} studied the counting function 
\[
N_3^\pm(X)=\#\{K \text{ cubic field }\mid 0<\pm \Delta_K<X\},
\]
and, using Delone--Faddeev correspondence \cite{df}, showed that 
\[
N_3^+(X)\sim \frac{1}{12\zeta(3)}X, \quad N_3^-(X)\sim \frac{1}{4\zeta(3)}X \quad \text{ as } X\rightarrow\infty.
\]
The first power saving error term was established by Belabas--Bhargava--Pomerance \cite{bbp}. The asymptotic formula was further strengthened in \cite{TT, bst, BTT} by incorporating the secondary term, which was conjectured by Roberts \cite{rob}, and an improved error term. 

To count $d\leq X$ in our Theorem \ref{main:thm}, we need to switch between counting cubic fields and cubic discriminants. Unlike quadratic fields, non-isomorphic cubic fields can have the same discriminant. For example, Taussky–Scholz \cite{ts} proved that there are four non-isomorphic cubic fields with discriminant $-3299$.
Thus, when counting cubic fields, we count discriminants with multiplicity.

By Hermite's Theorem, for any fixed integer $\Delta>1$, there are only finitely many number fields $K$ of degree $n$ with $|\Delta_K|=\Delta$ (see \cite[Section $4.1$]{ser}). Duke \cite[Section $3$]{duke} and Ellenberg–Venkatesh \cite[Conjecture $1.3$]{ev05} proposed the following conjecture regarding the number of fields with discriminant equal to $\Delta$: \textit{For each $n\geq 2$, for every $\varepsilon>0$, and for every integer $\Delta\geq 1$ there are $\ll_{n,\varepsilon}\Delta^\varepsilon$ number fields $K$ of degree $n$ with $|\Delta_K|=\Delta$.}

For $n=2$, this conjecture is trivially true, as there is at most one quadratic field of a given discriminant. However, for higher values of $n$, we only have partial results (for further exposition, see the article of Pierce--Turnage-Butterbaugh--Wood \cite{ptw}). For cubic fields, using \cite[Satz $7$]{Has} and the bound on the $3$-torsion part of the class group of quadratic fields, Ellenberg--Venkatesh \cite[Corollary $3.7$]{ev} showed that there are $\ll_\varepsilon |\Delta_K|^{1/3+\varepsilon}$ cubic fields $K$ of discriminant $\Delta_K$. If we employ the most recent bound of Chan--Koymans \cite{ck} on the $3$-torsion part of quadratic fields $\#\cl(\Q(\sqrt{d}))\ll_\varepsilon d^{\kappa+\varepsilon}$, then this implies that there are $\ll_\varepsilon|\Delta_K|^{\kappa+\varepsilon}$ cubic fields of discriminant $\Delta_K$.

Let us also mention that we can obtain better bounds on the number of $d\leq X$ in Theorem \ref{main:thm} by assuming some standard conjectures. Specifically, assuming  the Hasse–Weil conjecture,
the refined Birch and Swinnerton-Dyer conjecture, and the generalized Riemann hypothesis for $L$-functions
of elliptic curves over quadratic fields, Shankar–Tsimerman \cite[Theorem $1$]{st} showed that $\#\cl(\Q(\sqrt{d}))\ll_\varepsilon d^{\frac{1}{4}+\varepsilon}$. Therefore, by \cite[Satz $7$]{Has}, it follows that there are $\ll_\varepsilon \Delta_K^{\frac{1}{4}+\varepsilon}$ cubic fields with discriminant $\pm \Delta_K$. In this case, our proof yields that there are at least $cX^{\frac{3}{4}-\varepsilon}$ integers $d\leq X$ such that the assertions 1. and 2. of Theorem \ref{main:thm} hold.

\begin{notation}
    Throughout the paper, we will use the standard notation: for nonnegative functions $f(x),g(x)$, we write $f(x)=O(g(x))$ or $f(x)\ll g(x)$ if there is a constant $C$ such that $f(x)\leq C g(x)$ holds for all sufficiently large $x$. Further, we will use $f(x)=O_{\ell}(g(x))$ or $f(x)\ll_{\ell} g(x)$ to indicate that the implied constant $C$ depends on parameter(s) $\ell$. We write $f(x)=o(g(x))$ if $\lim_{x\rightarrow \infty}f(x)/g(x)=0$. 
\end{notation}

\section*{Acknowledgments}

We are very grateful to Frank Thorne for kindly answering our questions, as well as Subham Roy, Pavlo Yatsyna, and the rest of UFOCLAN group for our interesting and helpful discussions.

\section{Cubic fields and their class numbers}

Let $K$ be a cubic field, $\mco_K$ its ring of algebraic integers, and $h(K)$ its class number. 
The discriminant $\Delta_K$ of $K$ is of the form $df^2, 9df^2, \text{ or }81df^2$, where $d$ is a fundamental discriminant and $f$ is a squarefree positive integer coprime to $3$. More precisely, by \cite[Satz $6$]{Has}, if $\Delta_K$ is the discriminant of a cubic field, then $\Delta_K=df^2$ with $f=p_1\cdots p_n\cdot 3^w$, where $n\geq 0$, $w\in\{0,1,2\}$, and $p_i$ are distinct primes $\neq 3$ such that $$\left(\frac{d}{p_i}\right)\equiv p_i\pmod 3.$$ 
Recall also that $K/\Q$ is Galois if and only if $d=1$, i.e., $\Delta_K$ is a square.
 
One can characterize all  primes $p$ that are \textit{totally ramified} in $K$, i.e., $p\mco_K=\mathfrak{p}^3$ for a prime ideal  $\mathfrak{p}$  in $\mco_K$: $p\neq 3$ is {totally ramified} in $K$ if and only if $p$ divides $f$, and $p=3$ is totally ramified in $K$ if and only if $\Delta_K=9df^2 \text{ or }81df^2$, see \cite[Satz $5$]{Has} or \cite[Section $6.4.5$]{Coh}. 

\begin{lemma}[Hasse]\label{lem: p1mod3}
    Let $K$ be a cubic field of discriminant $\Delta_K=df^2$, and $p\neq2,3$ be totally ramified in $K$. Then $\left(\frac{d}{p}\right)=1$ if and only if $p\equiv 1\pmod 3$. 
\end{lemma}

\begin{proof}
    We refer the reader to \cite[Page $576, 577$]{Has} or \cite[Chapter $6$, Example $10$]{Ish}). 
\end{proof}
The \emph{genus field} of $K$ is defined as the maximal extension $K_{\text{gen}}$ of $K$ that is unramified at all finite primes and is compositum of $KF$, where $F$ is an abelian extension of $\Q$. The \textit{genus number} $g_K$ is defined as the degree $[K_{\text{gen}}:K]$. A theorem of Fröhlich gives a complete description of $g_K$ (see \cite{fro} or \cite[Chapter $6$, Example $10$]{Ish}).

\begin{theorem}[Fröhlich]\label{gen}
    Let $K$ be a cubic field and $e$ be the number of distinct odd primes $p$ such that 
 $p$ is totally ramified in $K$ and
 $\left(\frac{d}{p}\right)=1$.
 Then 
    \[
    g_K=\begin{cases}
        3^{e-1} \text{ if $K/\Q$ is Galois}\\
        3^e  \text{\ \ \   if $K/\Q$ is not Galois.}
    \end{cases}
    \]

\end{theorem}

We will further need a well-known result that relates the genus number with the class number.

 \begin{theorem}\label{divisibility}
     Let $K$ be a cubic field. Then $g_K$ divides $h(K)$.
 \end{theorem}

\begin{proof}
     It is well-known (e.g., see \cite[Paragraph $1$]{fro}) that $g_K$ divides \emph{narrow class number} $h^+(K)$. Since the class number and the narrow class number differ by a power of $2$, and $g_K$ is a power of $3$, it follows that $g_K$ divides $h(K)$. 
\end{proof}

In order to construct our families of consecutive cubic fields, we will work in the following setting that we fix in order to avoid repetition. Note that  $\varepsilon$ is different from the one in our main Theorem \ref{main:thm}.
 
\begin{setting S}\label{def: set s}
Fix a sign choice $\pm$, $\varepsilon>0$, and positive integers $k,H$. Let $n=\left\lceil\frac{\log H}{\log 3}\right\rceil+2$.

For all $1\leq i\leq k$, $1\leq j\leq n$, choose pairwise distinct primes  $p_{ij}\geq k$ and $q_i>\max(k, \left\lfloor \varepsilon^{-1}\right\rfloor+1)$ such that 
 $q_i\equiv1\pmod 4$ and 
 $p_{ij}\equiv 1 \pmod {12}$. Let $m=\prod_i q_i\prod_{i,j}p_{ij}^3$.

There is a unique integer $0\leq a<m$ such that
        \begin{equation}\label{A}
        \begin{aligned}
            a &\equiv q_i \prod_{j=1}^n p_{ij}^2 - i \pmod{q_i\prod_{j=1}^n p_{ij}^3}
        \end{aligned}
    \end{equation}
     for all $1\leq i\leq k$.
     Further, $p_{ij}^3\nmid a+i$ for all $i,j$, and $\frac{q_i-1}{q_i}>1-\varepsilon$ for all $i$.
\end{setting S}
    
\begin{proof}
By Chinese Remainder Theorem, there exists a unique solution $a\pmod m$ to \eqref{A}; we have
    \[
    a+i= q_i \prod_{j=1}^n p_{ij}^2 +t m, \ \  t\in \Z.
    \]
As $p_{ij}^3\mid m$, we see that $p_{ij}^3\nmid a+i$. Finally, since $q_i>\left\lfloor \varepsilon^{-1}\right\rfloor+1>\varepsilon^{-1}$ holds for all $i$; we have that $1-\frac{1}{q_i}>1-\varepsilon$.
\end{proof}

We can already establish the relevance of our setting for estimating class numbers.

\begin{proposition}\label{prop: largeclassno}
    Assume \hyperref[def: set s]{Setting (S)}. For $1\leq i\leq k$, if there exists a cubic field $K$ of discriminant $\Delta_K$ with $\Delta_K\equiv a+i \pmod m$, then $h(K)>H$.
\end{proposition}

\begin{proof}
    Suppose that there exists a cubic field $K$ of discriminant $\Delta_{K}$ with $\Delta_{K}\equiv a+i\pmod m$. Then $ \Delta_K=\pm q'_if_i^2$, where $f_i^2=\prod_{j=1}^n p_{ij}^2$ and $\pm q'_i=\pm q_i(1+t\prod_{i'\neq i}q_{i'}\cdot \prod_{j=1}^n p_{ij}\cdot \prod_{i'\neq i,j}p_{i'j}^3)$ is the fundamental discriminant. It is clear that all the primes $p_{i1},\dots,p_{in}$ divide $f_i$. Thus, for all $1\leq j\leq n$, primes $p_{ij}$ are totally ramified in $K$. From Lemma \ref{lem: p1mod3} and the choice $p_{ij}\equiv 1 \pmod {12}$, it follows that the fundamental discriminant $\pm q'_i$ also satisfies $\left(\frac{\pm q'_i}{p_{ij}} \right)=\left(\frac{\pm q_i}{p_{ij}} \right)=1$ for each $1\leq j\leq n$. Therefore, we can apply Theorem \ref{gen} to conclude that the genus number $g_{K}> 3^{n-1}>H$ because $n>\frac{\log H}{\log 3}+1$. Finally, as $g_{K}$ divides $h(K)$ by Theorem \ref{divisibility}, we have that $h(K)>H$. 
\end{proof}

\section{Counting cubic fields in an arithmetic progression}\label{s2}

We saw that, if there is a cubic field $K$ with discriminant $\pm\Delta_K$ among the consecutive integers $a+tm+1,\dots,a+tm+k$, then it has class numbers $h(K)>H$. However, it could easily happen that this statement is vacuous, as there could be no or very few such fields $K$.

To ensure the existence of our fields, we will use a result  of  Taniguchi--Thorne \cite[Theorem $6.2$]{TT} on the number of cubic fields in an arithmetic progression. In this section, we collect some necessary details from \cite{TT} and their helpful consequences. 

Fix a sign choice $\pm$. For $a, m \in \mathbb{Z}_{>0}$ and $X\in\R_{>0}$, define the counting function
\[
N_3^{\pm}(X; m, a) := \#\left\{ K \text{ cubic field } \mid 0 < \pm \Delta_K < X,\ \Delta_K \equiv a \pmod{m} \right\}.
\]
Here and in what follows, we are counting cubic fields up to isomorphism. There are always 1 or 3 isomorphic cubic fields (that are distinct as subfields of $\C$) depending on whether $K$ is Galois or not. 

Taniguchi–Thorne \cite[Equation 6.40]{TT} gave the following asymptotic formula:
    \begin{equation}\label{tt:asym}
        N_3^{\pm}(X;m,a) = C(m,a)\frac{C^{\pm}}{12\zeta(3)}  X + K_1(m,a)\frac{4K^\pm}{5\Gamma(2/3)^3}X^{5/6} + O(X^{7/9+\varepsilon}m^{8/9}),
    \end{equation}
where $C^-=3$, $C^+=1$, $K^-=\sqrt{3}$, and $K^+=1$. The constant $K_1(m,a)$ in the secondary term is quite complicated to state; thus, we refer the reader to \cite{TT} for a precise description of $K_1(m,a)$ (and, in fact, the error term can be improved further). 

The constant $C(m,a)$ in the main term is the density of discriminants of cubic fields in the arithmetic progression $a \pmod{m}$ as per \cite[Remark, Page $2498$]{TT} (note that \cite{TT} denoted this constant as $C_1(m,a)$).  In this paper, we are only interested in the question whether $C(m,a)>0$, because we see from \eqref{tt:asym} that $C(m,a)>0$ if and only if there exists $\delta>0$ such that $N_3^\pm (X;m,a)>\delta X$ holds for all sufficiently large $X$. 

In particular, when $\gcd(a,m) = 1$ and $m \not\equiv 0 \pmod{4}$, then we have \cite[Equation 6.41]{TT}
\begin{equation}\label{den:coprime}
    C(m,a)=\frac{1}{m} \prod_{p|m} \frac{1}{1-p^{-3}}>0.
\end{equation}

In general, the value of $C(m,a)$ is a bit more complicated: For a general arithmetic progression $ar$ modulo $mr$, where $\gcd(m,a)=1$ and $m,r$ need not be coprime, we can express the density $C(m,a)$ as a product of local densities at primes dividing $mr$ \cite[Page $2501$, Paragraph $1$]{TT}. These local densities can be computed using \cite[Proposition $6.4$]{TT}.

\begin{proposition}[{\cite[Proposition $6.4$]{TT}}]\label{p:tt}
    Let $p>3$ be a prime and $a$ be coprime to $p$. Then we have 
    \begin{equation}\label{den:loc}
        C(p^2,ap)= \frac{1}{p^2(1-p^{-3})}, \ \ 
        C(p^3,ap^2)= \frac{1+\phi_p(-3a)}{p^3(1-p^{-3})}, 
    \end{equation}
    where $\phi_p(-3a)=\left(\frac{-3a}{p}\right)$ is the Legendre symbol. 
\end{proposition}
Since for $p>3$, there are no cubic discriminants divisible by $p^3$, the local densities stated above are sufficient for computing $C(m,a)$. Note that $C(p^2,ap)$  never equals zero, but $C(p^3,ap^2)=0$ exactly when $-3a$ is a quadratic non-residue modulo $p$, for instance, $C(7^3, 3\cdot 7^2)=0$ (see \cite[Page $2503$]{TT}).

For the sake of convenience, we now establish several explicit estimate on the densities which follow immediately from \cite{TT}.

\begin{lemma}\label{l2}
Fix a sign choice $\pm$, $a, k\in\Z_{>0}$, and $\varepsilon>0$. Let $m=\prod_{i=1}^k p_i^{r_i}$, where $p_1,p_2,\dots,p_k$ are distinct primes $p_i>\max\left(3,\left\lfloor\varepsilon^{-1}\right\rfloor+1\right)$ and $r_i\in\{1,2,3\}$. For each $1\leq i\leq k$, suppose that $C(p_i^{r_i}, a)>0$. Then, we have 
\[
C(m,a)>\frac{1}{m}\left(1-\varepsilon \right)^t,
\]
where $t=\#\{i\mid r_i=1 \}$.
\end{lemma}
\begin{proof}
From \cite[Page $2501$, Paragraph $1$]{TT} and the assumption that $C(p_i^{r_i}, a)>0$, it follows that the density
    \begin{equation}\label{prod}
        C\left(m, a\right) = \prod_{i=1}^k C(p_i^{r_i}, a)>0.
    \end{equation}

From \eqref{den:coprime} and \eqref{den:loc}, we have that $C(p_i^{r_i}, a)\geq \frac{1}{p_i^{r_i}(1-p_i^{-3})}$ holds if $r_i\in\{2,3\}$, or if $r_i=1$ and $\gcd(a,p_i)=1$. 

To control $C(p_i, 0)$, we first observe that 
    \begin{equation}\label{est.for0}
        C(p_i, 0)\geq \sum_{j=1}^{p_i-1}C(p_i^2, jp_i)= \frac{p_i-1}{p_i^2(1-p_i^{-3})}.
    \end{equation}
 Thus, we have 
    \begin{equation}\label{localbound}
        C(p_i^{r_i}, a)\geq \begin{cases}
            \frac{1}{p_i^{r_i}(1-p_i^{-3})} \textnormal{ if } r_i\in\{2,3\}\\
            \frac{p_i-1}{p_i^2(1-p_i^{-3})} \textnormal{ if }r_i=1.
        \end{cases}
    \end{equation}
    
Let 
    \[
    \psi(p_i^{r_i})=\begin{cases}
        1 \textnormal{ if } r_i\in\{2,3\}\\
        \frac{p_i-1}{p_i} \textnormal{ if } r_i=1.
    \end{cases}
    \]
Now, using the bound \eqref{localbound} on local density in \eqref{prod}, we see that 
    \begin{equation*}
        C(m,a)\geq \prod_{i=1}^k \frac{\psi(p_i^{r_i})}{p_i^{r_i}(1-p_i^{-3})}\geq \frac{1}{m}\prod_{i=1}^k\frac{1}{1-p_i^{-3}}\prod_{i=1}^k\psi(p_i^{r_i}).
    \end{equation*}

Since each $p_i>\left\lfloor\varepsilon^{-1}\right\rfloor+1$, we get
    \begin{equation*}
        C(m,a)> \frac{1}{m}\prod_{i=1}^k\psi(p_i^{r_i})\geq \frac{1}{m}\left(1-\varepsilon \right)^t,
    \end{equation*}
    where $t=\#\{i\mid r_i=1 \}$.
\end{proof}

We now show that for each $i$, cubic fields with discriminant in our arithmetic progression $\Delta_{K}\equiv a+i \pmod m$ have reasonably high, positive density.

\begin{proposition}\label{prop: rowdensity}
Assume \hyperref[def: set s]{Setting (S)}. Then $C(m,a+i)>\frac{1-\varepsilon}{m}$ holds for all $1\leq i\leq k$.
\end{proposition}

\begin{proof}
Fix an arbitrary integer $1\leq i\leq k$. 
We want to use Lemma \ref{l2} for the factorization 
$m=\prod_h q_h\prod_{h,j}p_{hj}^3$, and so we first need to show that $C(q_h,a+i)>0, C(p_{hj}^3,a+i)>0$ for all $h,j$.

First, let $h\neq i$. We claim that $\gcd(q_h,a+i)=1=\gcd(p_{hj}^3,a+i)$. If not, suppose $q_h\mid a+i$. As $q_h\mid a+h$, we get that $q_h\mid h-i$, which is impossible because $q_h> k$; the case of $p_{hj}$ is analogous. Now, from \eqref{den:coprime}, it follows that in these coprime arithmetic progressions the densities are
\begin{equation*}
    C(q_h,a+i)= \frac{1}{q_h(1-q_h^{-3})}>0, \textnormal{ and } \ \ 
    C(p_{hj}^3,a+i)= \frac{1}{p_{hj}^3(1-p_{hj}^{-3})}>0
\end{equation*}
as we want.

As $a+i \equiv 0 \pmod{q_i}$, \eqref{est.for0} implies that $C(q_i,a+i)=C(q_i,0)>0$.
    
    Finally, let us show that for each $1\leq j\leq n$, the density $C(p_{ij}^3, a+i)>0$. 
By Proposition \ref{p:tt}, we have that 
    \[
    C(p_{ij}^3, a+i)=C\left(p_{ij}^3, q_i \prod_{j=1}^n p_{ij}^2 \right)= \frac{1+\phi_{p_{ij}}(-3q_ip_{i1}^2\dots p_{i{j-1}}^2p_{i{j+1}}^2\dots p_{in}^2)}{p_{ij}^2(1-p_{ij}^{-3})}.
    \]
    To show $C(p_{ij}^3, a+i)>0$, we thus want to show that the Legendre symbol satisfies
    \begin{align*}
    1&=\phi_{p_{ij}}(-3q_ip_{i1}^2\dots p_{i{j-1}}^2p_{i{j+1}}^2\dots p_{in}^2)\\
    &=\left(\frac{-3q_ip_{i1}^2\dots p_{i{j-1}}^2p_{i{j+1}}^2\dots p_{in}^2}{p_{ij}}\right)=\left(\frac{-3q_i}{p
    _{ij}}\right)\\
    &=\left(\frac{-1}{p_{ij}}\right)\left(\frac{3}{p_{ij}}\right)\left(\frac{q_i}{p_{ij}}\right).
    \end{align*}
    
    The prime $p_{ij}$ is totally ramified in the cubic field of discriminant $\Delta_K$ satisfying $\Delta_K\equiv q_i \prod_{j=1}^n p_{ij}^2 \pmod{p_{ij}^3}$. Since $p_{ij}\neq 2,3$, from Lemma \ref{lem: p1mod3} it follows that $q_i$ is quadratic residue modulo $p_{ij}$. Thus, we have $\left(\frac{q_i}{p_{ij}}\right)=1$. As $p_{ij}\equiv 1 \pmod{12}$, we have
    $\left(\frac{-1}{p_{ij}}\right)=1$,  and by quadratic reciprocity $\left(\frac{3}{p_{ij}}\right)=\left(\frac{p_{ij}}{3}\right)=1$. 

Now, to estimate the density, we apply Lemma \ref{l2} for our $m=\prod_h q_h\prod_{h,j}p_{hj}^3$ to get $C(m,a+i)>\frac{1}{m}\left(1-\varepsilon\right)^t$. Since only $q_i\|a+i$, we have that $C(m,a+i)>\frac{1-\varepsilon}{m}$.
\end{proof}

\section{Density estimates}

In this section we prove our main theorem. To do that, we will need to transition between counting fields and  counting discriminants. 

\begin{theorem}\label{multbound}
    For every $\varepsilon>0$, and for every positive integer $\Delta> 1$ there are $\ll_\varepsilon \Delta^{\kappa+\varepsilon}$ cubic fields $K$ with $\Delta_K=\Delta$.
\end{theorem}
\begin{proof}
    This follows from the work of Hasse \cite[Satz $7$]{Has} using Chan--Koymans \cite{ck} bound $\#\cl(K)\ll_\varepsilon |\Delta_K|^{\kappa+\varepsilon}$ for all quadratic fields $K$. Recall that, throughout the paper, $\kappa\approx 0.3193\dots$ is the constant from \cite{ck}.
\end{proof}

\begin{theorem}\label{thm:estimate}
    Fix a sign choice $\pm$, $k\in\Z_{>0}$, and $\varepsilon>0$. Suppose that for integers $0\leq a<m$, there exists a $\delta>0$ such that $N_3^\pm(X;m,a+i)>2\delta X$ holds for all $1\leq i\leq k$ and for all sufficiently large $X$. Then, there exists a positive constant $c$ (depending on $\pm, \varepsilon, a, m,k$) such that for all sufficiently large $X$, there are at least $c\delta X^{1-\kappa-\varepsilon}$ integers $t$ with $0\leq t\leq X$ such that 
     $$\#\{K \text{ cubic field }\mid \pm \Delta_K\in\{a+tm+1,a+tm+2,\dots,a+tm+k\}\}> \delta km.$$ 
\end{theorem}

\begin{rk}
    Note that the number of cubic fields among $a+tm+1,a+tm+2,\dots,a+tm+k$ could be small when $\delta$ is small. However, in our main theorem we use specific $\delta$ obtained in Proposition \ref{prop: rowdensity}, that will ensure that there are many fields.
\end{rk}

\begin{proof}
    We use a variant of Maier matrix method \cite{maier} (see also  \cite{shiu} or the survey articles \cite{th,ag}). For a positive integer $X$, consider the matrix
    \[
    \begin{bmatrix}
        a+1 & a+2 & a+3 & \cdots & a+k \\
    a+m+1 & a+m+2 & a+m+3 & \cdots & a+m+k \\
    a+2m+1 & a+2m+2 & a+2m+3 & \cdots & a+2m+k \\
    \vdots & \vdots & \vdots & \ddots & \vdots \\
    a+Xm+1 & a+Xm+2 & a+Xm+3 & \cdots & a+Xm+k 
    \end{bmatrix}
    \]
    of integers. As a notational shorthand, we will say that a subset $S$ of integers in the matrix contains $>B$ cubic fields if $\#\{K \text{ cubic field }\mid \pm\Delta_K\in S\}>B$. An entry in the matrix will be called a \emph{good entry} if it is a cubic discriminant. 
    
    By our assumption that there exists a $\delta>0$ such that $N_3^\pm(X;m,a+i)>2\delta X$ holds for all $1\leq i\leq k$ and for all sufficiently large $X$. It follows that 
    \[
    N_3^\pm(a+Xm+i+1;m,a+i)>2\delta (a+Xm+i+1)>2\delta mX
    \]
    holds for all $1\leq i\leq k$ and for all sufficiently large $X$. That is, for all sufficiently large $X$, each column  contains $>2\delta mX$ cubic fields. Thus, in total, there are $>2\delta kmX$ cubic fields in the matrix.
    
      We now claim that for all sufficiently large $X$, there exists a row that contains $>2\delta km$ cubic fields. Suppose for contradiction that each row contains $\leq 2\delta km$ cubic fields. Then it follows that the total number of cubic fields in the matrix is $\leq 2\delta kmX$, which leads to a contradiction with the fact that there are $>2\delta kmX$ cubic fields in the matrix.

    Let us call a row containing $>\delta km$ cubic fields a \emph{good row}, and let us denote by $G$ the total number of good rows. Thanks to the above argument, for all sufficiently large $X$, we know that there exists at least one good row in the matrix, i.e., $G\geq1$.  

    Now, for a sufficiently large $X$, our aim is to count the number of  good rows. That is, we want to count integers $t$ with $0\leq t\leq X$ such that the row $a+tm+1,a+tm+2,\dots, a+tm+k$ is a good row. For a given $\varepsilon>0$, by Theorem \ref{multbound} there exists a positive constant $c_1$ (depending on $\pm, \varepsilon$) such that for each good entry $\pm(a+tm+i)$ there are $<c_1(a+tm+i)^{\kappa+\varepsilon}$ cubic fields with discriminant $\pm(a+tm+i)$. Since the number of good entries in a row $\leq k$, we see that a good row contains $< c_1k(a+Xm+k)^{\kappa+\varepsilon}$ cubic fields. As there are $>2\delta kmX$ cubic fields in the matrix, we have $$2\delta kmX<c_1Gk(a+Xm+k)^{\kappa+\varepsilon}+(X-G)\delta km.$$ Then, it follows that
    \begin{equation*}
    \begin{aligned}
        \delta kmX<Gk (c_1(a+Xm+k)^{\kappa+\varepsilon}-\delta m)
        <Gk (c_1(a+Xm+k)^{\kappa+\varepsilon})< Gkc_2mX^{\kappa+\varepsilon},
    \end{aligned}
    \end{equation*}
    where $c_2$ (which depends on $\pm, \varepsilon, a, m, k$) is a positive constant. Thus, we have $G>c\delta X^{1-\kappa-\varepsilon}$, where $c=1/c_2$. This completes the proof.
\end{proof}

\begin{theorem}\label{thm:final}
    Assume \hyperref[def: set s]{Setting (S)}.  There exists a positive constant $c$ (depending on $\pm, \varepsilon, k, H$) such that for all sufficiently large $X$, there are at least $cX^{1-\kappa-\varepsilon}$ positive integers $t\leq X$ such that the following hold:
    \begin{enumerate}
        \item[(1)] For each $1\leq i\leq k$, if there exists a cubic field $K$ with $\pm\Delta_K=a+tm+i$, then $h(K)>H$.
        \item[(2)] We have
        \[
        \#\{K \text{ cubic field }\mid \pm \Delta_K\in\{a+tm+1,a+tm+2,\dots,a+tm+k\}\}\geq \frac{k}{2}.
        \]
    
    \end{enumerate}
\end{theorem}

\begin{proof} 
    \textit{(1)} From Proposition \ref{prop: largeclassno}, we conclude that for $1\leq i\leq k$, if there exists a cubic field $K$ with $\pm\Delta_K=a+tm+i$, then it has class number $>H$.
    
     By Proposition \ref{prop: rowdensity}, we have that $N_3^\pm(X; m, a+i)>\left(\frac{1-\varepsilon'}{m} \right) X$ holds for all $\varepsilon'>0$, all $1\leq i\leq k$, and all sufficiently large $X$.
    Now, by Theorem \ref{thm:estimate}, for all sufficiently large $X$, there are at least $cX^{1-\kappa-\varepsilon}$ integers $t$ with $0\leq t\leq X$ such that 
     $$\#\{K \text{ cubic field }\mid \pm \Delta_K\in\{a+tm+1,a+tm+2,\dots,a+tm+k\}\}> \left(\frac{1-\varepsilon'}{2m} \right) km=\left(\frac{1-\varepsilon'}{2} \right) k,$$ 
     where $c=c'\cdot \left(\frac{1-\varepsilon'}{m} \right)$, and  $c'$ is the constant from Theorem \ref{thm:estimate}.   
     Choosing $\varepsilon'$ sufficiently small relative to $1/k$, we obtain \textit{(2)}.
\end{proof}

This establishes Theorem \ref{main:thm}. Note that it is not hard to optimize the choices of $\varepsilon, q_i,p_{ij}$ in \hyperref[def: set s]{Setting (S)} in order to obtain an explicit upper bound on $a$ and $m$ in Theorem \ref{thm:final}. However, in absence of an explicit bound on the constant $c$ (that comes from Theorem \ref{multbound} and \cite{ck}), this would not give us an explicit statement on the admissible integers $d$ in Theorem \ref{main:thm}.

\bibliographystyle{amsalpha}
	
\bibliography{Citation} 

\end{document}